\documentclass{article}

\usepackage{mathtools,amsthm,amsmath,amsfonts,amssymb,euscript,verbatim,tikz,hyperref}

\theoremstyle{plain}
 \newtheorem{thm}{Theorem}[section]
 \newtheorem{prop}[thm]{Proposition}
 \newtheorem{lemma}[thm]{Lemma}
 \newtheorem{cor}[thm]{Corollary}

\newcommand{\mbR}{\mathbb{R}}
\newcommand{\mbC}{\mathbb{C}}
\newcommand{\mbZ}{\mathbb{Z}}

\newcommand{\euE}{{\EuScript E}}

\newcommand{\euH}{{\EuScript H}}

\newcommand{\clA}{\mathcal{A}}
\newcommand{\clB}{\mathcal{B}}
\newcommand{\clH}{\mathcal{H}}
\newcommand{\clK}{\mathcal{K}}
\newcommand{\clL}{\mathcal{L}}
\newcommand{\clU}{\mathcal{U}}

\newcommand{\Tr}{\operatorname{Tr}}
\renewcommand{\Im}{\operatorname{Im}}

\newcommand{\hilb}{{\mathcal H}} 
\newcommand{\hlambda}{{\mathfrak{h}_\lambda}}
\newcommand{\dom}{\operatorname{dom}} 
\newcommand{\rng}{\operatorname{ran}}
\newcommand{\scal}[1]{\left\langle #1 \right\rangle}
\newcommand{\Texp}{\mathrm{T}\!\exp}
\newcommand{\xia}{\xi^{(a)}}
\newcommand{\xis}{\xi^{(s)}}

\begin{document}
\title{Singular spectral shift function \\ for Schr\"odinger operators}
\author{N.\,Azamov and T.\,Daniels}

\maketitle
\begin{center}
{
\small Flinders University	\\ \footnotesize School of Computer Science, Engineering and Mathematics
\\ 1284 South Road, Clovelly Park, 5042, SA, Australia
\\ \href{mailto:nurulla.azamov@flinders.edu.au}{\nolinkurl{nurulla.azamov@flinders.edu.au}}
\\ \href{mailto:tom.daniels@flinders.edu.au}{\nolinkurl{tom.daniels@flinders.edu.au}}
}
\end{center}

\begin{abstract}
Let $H_0= -\Delta + V_0(x)$ be a Schr\"odinger operator on $L_2(\mbR^\nu),$ $\nu=1,2,$ or~$3,$ where $V_0(x)$ is a bounded measurable real-valued function on $\mbR^\nu.$ 
Let $V$ be an operator of multiplication by a bounded integrable real-valued 
function $V(x)$ and put $H_r = H_0+rV$ for real $r.$
We show that the associated spectral shift function (SSF) $\xi$ admits a natural decomposition into the sum of absolutely continuous $\xi^{(a)}$ and singular $\xi^{(s)}$ SSFs. In particular,
\[
 \xi^{(s)}(\varphi) := \int_0^1 \Tr\left(E^{(s)}_{H_r}(\mathrm{supp}\varphi)V \varphi(H_r)\right)\,dr, \quad \varphi \in C_c(\mbR),
\]
where $E^{(s)}_{H}$ denotes the singular spectral measure of $H,$
defines an absolutely continuous measure whose density $\xi^{(s)}(\lambda)$ (denoted by the same symbol) is integer-valued for a.e. $\lambda.$

This is a special case of an analogous result for resolvent comparable pairs of self-adjoint operators, which generalises the case of a trace class perturbation appearing in \cite{Az3v6} while also simplifying its proof.
We present two proofs -- one short and one long -- which we consider to have value of their own.
The long proof along the way reframes some classical results from the perturbation theory of self-adjoint operators, including the existence and completeness of the wave operators and the Birman-Krein formula relating the scattering matrix and the SSF.
The two proofs demonstrate the equality of the singular SSF with two a priori different but intrinsically integer-valued functions: the total resonance index \cite{Az9} and the singular $\mu$-invariant \cite{Az3v6}.
\end{abstract}

\vspace*{\fill}
\footnoterule
\vskip 2mm
{

}

\section{Introduction}
The spectral shift function (SSF) $\xi(\lambda)=$ $\xi(\lambda; H_1,H_0)$ for a pair of self-adjoint operators $H_0$ and $H_1$ on a Hilbert space $\hilb$ is a real-valued function which satisfies the Lifshitz-Krein trace formula \cite{Li52UMN,Kr53MS}
\begin{equation}\label{F: Kreins trace formula}
  \Tr(\varphi(H_1) - \varphi(H_0)) = \int_{-\infty}^\infty \varphi'(\lambda)\xi(\lambda)\,d\lambda
\end{equation}
for all test functions $\varphi \in C_c^\infty(\mbR),$ provided that the left hand side makes sense.
Adopting terminology which appears in \cite{Ya}, we will say that $H_0$ and $H_1$ are {\em resolvent comparable} if
\begin{equation}\label{F: resolvent comparable}
  R_z(H_1) - R_z(H_0) \in \clL_1(\hilb),
\end{equation}
for some (hence any) nonreal $z,$ where $\clL_1(\hilb)$ is the trace class and $R_z(H) = (H-z)^{-1}$ is the resolvent of $H.$
Famously, M.\,G.\,Krein established the existence of the SSF assuming a trace class difference $V := H_1 - H_0$ in \cite{Kr53MS}, and later extended this to resolvent comparable pairs by transformation from the case of unitary pairs with trace class difference in \cite{Kr62DAN}.

Historically the trace formula is used to define the SSF, although this in general leaves an additive constant unspecified.
Another celebrated development appeared in \cite{BS75SM}, where M.\,Sh.\,Birman and M.\,Z.\,Solomyak proved that if the perturbation operator $V$ is trace class then the SSF is the density of an absolutely continuous measure $\xi$ given by the formula 
\begin{equation} \label{F: B-S formula}
  \xi(\varphi) = \int_0^1 \Tr\left(E_{H_r}(\mathrm{supp}\varphi) V \varphi(H_r)\right)\,dr, \quad \varphi \in C_c(\mbR),
\end{equation}
where $H_r = H_0 + rV$ and $E_H$ denotes the spectral measure of the self-adjoint operator $H$ (which appears in this formula so that it can later be discussed in a more general context without modification).
Here and in what follows we are identifying locally finite Borel measures with continuous linear functionals on the linear space of continuous functions with compact support $C_c(\mbR),$ endowed with the standard locally convex inductive topology.
Birman and Solomyak's proof of formula \eqref{F: B-S formula} using the theory of double operator integrals is simple and natural, but it does not show the absolute continuity of the spectral shift measure $\xi.$ 
Nevertheless, we will consider the Birman-Solomyak formula as the definition of the spectral shift measure. 
From this point of view the argument of \cite{BS75SM} can be considered as a proof of the trace formula, while Krein's original argument can be considered as a proof of absolute continuity.

An indication of the fundamental nature of the Birman-Solomyak formula is that it represents the spectral shift measure as an integral of the generalised 1-form
\begin{equation}\label{F: ISSM defn}
  V \mapsto \Phi_H(V)(\varphi) := \Tr(E_H(\mathrm{supp}\varphi)V\varphi(H))
\end{equation}
on some real affine space of self-adjoint operators. 
In \cite{AS2} it is shown that this one-form, the {\em infinitesimal spectral shift measure,} is exact on the corresponding affine space as long as the product $V\varphi(H)$ is trace class for test functions $\varphi.$ 
Then the independence of the spectral shift measure \eqref{F: B-S formula} from the path $H_r$ is used to reduce its absolute continuity to the case of a trace class perturbation.

While the spectral shift measure $\xi$ is absolutely continuous, the infinitesimal spectral shift measure $\Phi$ is not. 
Therefore, as in \cite{Az3v6} one can consider the decomposition of $\Phi$ into its absolutely continuous $\Phi^{(a)}$ and singular $\Phi^{(s)}$ parts.
The absolutely continuous component $\Phi^{(a)}$ is absent outside the common essential spectrum $\sigma_{ess}$ of the operators $H_j,$ $j=0,1.$ 
Thus outside $\sigma_{ess},$ properties of $\Phi$ are properties of its singular component $\Phi^{(s)}.$ 
The same cannot be said inside $\sigma_{ess},$ where in particular the exactness of $\Phi^{(s)}$ fails.
However one of the crucial properties of the SSF outside $\sigma_{ess},$ its integer-valuedness, can be retained inside $\sigma_{ess}$ if the SSF is replaced by the {\em singular SSF,} which is the (path-dependent) integral of $\Phi^{(s)}:$
\begin{equation} \label{F: xis defn}
  \xis(\varphi) = \int_0^1 \Tr\left(E^{(s)}_{H_r}(\mathrm{supp}\varphi)V \varphi(H_r)\right)\,dr, \quad \varphi \in C_c(\mbR),
\end{equation}
where $E^{(s)}_{H}$ denotes the singular spectral measure of $H.$
In \cite{Az3v6} this is proved under the hypothesis that the perturbation operator $V$ is trace class. 
In this paper we prove the following generalisation of this result.

\begin{thm}\label{T: xis in Z}
Let $H_0$ be a self-adjoint operator on a Hilbert space $\hilb,$ let $V$ be a bounded self-adjoint operator on $\hilb,$ and let $H_r = H_0 + rV,$ $r\in\mbR.$
Assume that $H_0$ and $H_0+|V|$ are resolvent comparable (in which case so are $H_0$ and $H_1$). 
Define the singular spectral shift measure $\xis(\varphi) = \xis(\varphi; H_1,H_0)$ by the formula \eqref{F: xis defn}.
Then this measure is absolutely continuous and its density $\xis(\lambda),$ the singular SSF, takes integer values for a.e. $\lambda.$  
\end{thm}

As is well-known, on the Hilbert space $L_2(\mbR^\nu)$ for $\nu=1,2,$ or 3 (but not 4 or larger) the premise of Theorem \ref{T: xis in Z} holds for a pair of operators $H_0 = -\Delta + V_0$ and $V,$ where $V_0$ is the operator of multiplication by a bounded measurable real-valued function $V_0(x)$ on $\mbR^\nu$ and $V$ is the operator of multiplication by a bounded integrable real-valued function $V(x).$ 
Hence, Theorem \ref{T: xis in Z} includes the theorem appearing in the abstract.

\medskip
Two proofs of Theorem \ref{T: xis in Z} are presented below.
These proofs are relatively independent, but each uses the Limiting Absorption Principle, which we formulate as follows.
Let $F$ be a bounded injective operator, a {\em rigging operator,} from the main Hilbert space $\hilb$ to an auxiliary Hilbert space $\clK$ which we assume to be the closure of the range of $F$ (and is only for convenience different from $\hilb).$ 
Because a rigging operator will usually be fixed, dependence on it is often omitted from notation and terminology. 
For a fixed rigging $F,$ the sandwiched resolvent of a self-adjoint operator $H$ on $\hilb$ will be denoted
\[
  T_{\lambda + iy}(H) := F R_{\lambda + iy}(H) F^*.
\]
By $\Lambda(H,F)$ we denote the set of all real numbers $\lambda$ for which the limit $T_{\lambda + i0}(H)$ exists in $\clB(\clK)$ -- the bounded operators on $\clK$ with the usual norm. 
We say that the Limiting Absorption Principle holds for $H$ if $\Lambda(H,F)$ is a full set.

The premise of Theorem \ref{T: xis in Z} implies that the perturbation operator $V$ admits a decomposition $V = F^*JF,$ where $F$ is a rigging operator and $J$ is a bounded self-adjoint operator on the auxiliary Hilbert space, such that the operators $FR_z(H_r),$ for real $r$ and nonreal $z,$ belong to the Hilbert-Schmidt class
(which can be shown using the resolvent identities and choosing $F:=\sqrt{|V|} + F_k$ where $F_k$ is an injective Hilbert-Schmidt operator on the kernel of $V).$
It is rather the latter condition that will be used in the proofs.
(By adding to this condition the assumption that $F$ is self-adjoint and commutes with $J$ one can show equivalence with the premise of Theorem \ref{T: xis in Z}.)
This condition implies the Limiting Absorption Principle for each $H_r.$
Moreover, it implies that for each $H_r$ there is a full subset, denoted $\Lambda(H_r,F;\clL_1),$ of points $\lambda\in\Lambda(H_r,F)$ at which in addition the limit $\Im T_{\lambda + i0}(H_r)$ exists in the trace class $\clL_1(\clK).$
These last two implications are simple corollaries of classical results due to M.\,S.\,Birman and S.\,B.\,Entina \cite{BE} (also see \cite[Theorems 6.1.9 and 6.1.5]{Ya}), which make the same conclusions in the case that $F$ is itself Hilbert-Schmidt.

Before turning to the differences between the two proofs, we take this opportunity to say a little more about the sets $\Lambda(H_r,F;\clL_p),$ $p=1,$ or $\infty$ (in which case we mean $\Lambda(H,F;\clL_\infty) = \Lambda(H,F)).$
Their union for a collection of operators $\{H_r\}$ will be denoted $\Lambda(\{H_r\},F;\clL_p)$ and if $p=\infty$ its elements are called {\em essentially regular points} of the collection.
Suppose $H_r = H_0 + V_r,$ $r\in\mbR,$ where $V_r = F^*J_rF$ for some $H_0$-compact rigging $F.$ 
Then the resolvent identities imply the following identities for the sandwiched resolvent and its imaginary part.
\begin{align*}
 T_z(H_r) &= (1 + T_z(H_0)J_r)^{-1}T_z(H_0), 
 \\ \Im T_z(H_r) &= (1 + T_{\bar{z}}(H_0)J_r)^{-1}\Im T_z(H_0)(1 + J_rT_z(H_0))^{-1}. 
\end{align*}
The analytic Fredholm alternative implies that, if $J_r$ depends on $r$ analytically, the inverted factors are meromorphic functions of $r.$
Moreover, by considering the limit $z = \lambda+i0,$ it can be concluded 
that if $\lambda$ belongs to $\Lambda(H_0,F;\clL_p),$ then it also belongs to $\Lambda(H_r,F;\clL_p)$ for exactly all real numbers $r$ except the discrete set of real poles of the function $r\mapsto (1 + T_{\lambda+i0}(H_0)J_r)^{-1},$ or equivalently of the function $r\mapsto T_{\lambda+i0}(H_r).$ 
These real poles are called {\em resonance points} and their collection, the {\em resonance set}, is denoted $R(\lambda;\{H_r\}) = R(\lambda;\{H_r\},F).$ 

We will now briefly describe an alternate interpretation of the integer $\xis(\lambda),$ namely the total resonance index.
Let $\lambda$ be an essentially regular point of a path of self-adjoint operators $H_r = H_0 + rV,$ where $V = F^*JF$ for a $H_0$-compact rigging operator $F.$ 
Then at a resonance point $r_\lambda$ of this path, the {\em resonance index} is the number $N_+ - N_-,$ where $N_\pm$ is the number of poles of $r \mapsto T_{\lambda+iy}(H_r)$ in $\mbC_\pm,$ $y>0,$ which converge to $r_{\lambda}$ as $y\to 0^+.$ 
The {\em total resonance index} of the pair $H_0,H_1$ at $\lambda$ is the finite sum of resonance indices of those resonance points from the interval $[0,1].$
Showing that the singular SSF is equal to the total resonance index, following an argument appearing in \cite{Az9}, constitutes the shorter of the two proofs of Theorem \ref{T: xis in Z}.

The longer proof proceeds, following \cite{Az3v6}, via an ordered exponential representation of the scattering matrix $S(\lambda;H_1,H_0),$
to prove both the Birman-Krein (\cite{BK}) formula
\begin{equation} \label{F: det S = exp(-2pi i xi)}
  \det S(\lambda;H_1,H_0) = e^{-2\pi i \xi(\lambda)}
\end{equation} 
and its variant:
\begin{equation} \label{F: det S = exp(-2pi i xia)}
  \det S(\lambda;H_1,H_0) = e^{-2\pi i \xia(\lambda)}.
\end{equation} 
Here $\xia(\lambda)$ is the {\em absolutely continuous SSF,} which we identify with the measure given by the formula below and shown to be absolutely continuous in Section \ref{S: sSSF}. 
\begin{equation}\label{F: a.c.SSM}
  \xia(\varphi) = \int_0^1 \Tr\left(E^{(a)}_{H_r}(\mathrm{supp}\varphi)V \varphi(H_r)\right)\,dr,  \quad \varphi \in C_c(\mbR),
\end{equation}
where $E^{(a)}_H$ denotes the absolutely continuous spectral measure of $H.$
Since $\xi(\lambda) = \xia(\lambda) + \xis(\lambda),$ by combining \eqref{F: det S = exp(-2pi i xi)} and \eqref{F: det S = exp(-2pi i xia)} we can conclude that $\xis(\lambda)$ must be an integer.
This approach also allows us to establish the equality of the singular SSF and the singular $\mu$-invariant, which measures the difference in winding numbers of the eigenvalues of the scattering matrix $S(\lambda; H_1, H_0)$ as it is continuously deformed to the identity operator in two different and natural ways:
the first way is to send the imaginary part of $\lambda$ from 0 to $+\infty$ and the second way is to send $H_1$ to $H_0.$

Establishing an ordered exponential representation of the scattering matrix takes some work and gives the long proof its length.
In particular it requires, for a fixed value of the spectral parameter $\lambda,$ to be able to vary the coupling constant $r$ in objects of stationary scattering theory such as the scattering matrix $S(\lambda;H_r,H_0).$
To this end, defining $S(\lambda;H_r,H_0)$ via an arbitrary direct integral decomposition is not suitable because it results in modulo-null-set uncertainty.
For trace class perturbations, a workable approach to stationary scattering theory was presented in \cite{Az3v6}. 
This approach is modified in Section \ref{S: SST} and it is an aim of this paper to promote this simple approach.

\section{Spectral shift function (SSF)}\label{S: SSF}
In this section we give a brief exposition of the SSF for resolvent comparable self-adjoint operators from the point of view discussed above. 
We follow \cite{AS2}, adjusting the argument to our setting.
Let $\clH$ be a Hilbert space and suppose $\clA$ is a real affine space of self-adjoint operators on $\clH$ over a real vector space $\clA_0$ of perturbations, such that $\clA_0$ consists of bounded self-adjoint operators and any two operators from $\clA$ are resolvent comparable.
If $\clA_0$ is infinite-dimensional, we equip it with a norm, call it $\|\cdot\|_F,$ such that for some (hence any) $H\in\clA$ and some (hence any) nonreal $z,$ $\|V\|_F \geq \mathrm{const.}\|R_z(H)VR_z(H)\|_1$  for any $V\in\clA_0$ (this situation occurs naturally in the context of Section \ref{S: sSSF} below).
This ensures the following.
\begin{lemma}
If a path $H_r$ is $C^1$ in $\clA = H_0 + \clA_0,$ then the path $r\mapsto f(H_r)$ is $C^1$ in $f(H_0) + \clL_1(\clH)$ for any test function $f,$ and also $f=R_z$ for nonreal $z.$
\end{lemma}
\begin{proof}[Sketch]
This follows for $f=R_z$ from the assumption on the norm of $\clA_0$ and the second resolvent identity.
Then for test functions it can be shown with the aid of the Helffer-Sj\"ostrand formula.
\end{proof}

The Helffer-Sj\"ostrand formula (see e.g. \cite[Theorem 8.1]{DimSjo})
\[
 \varphi(H) = \frac{1}{\pi}\int_{\mbR^2} \bar{\partial}\tilde{\varphi}(z) R_z(H)\, dx dy
\]
provides a convenient way of treating double operator integrals in this context (and is useful for the previous and the next two proofs). 
In particular it implies the trace class integral formula
\begin{equation}\label{F: H-S DOI}
\frac{d}{dr}\varphi(H_r) =  - \frac{1}{\pi}\int_{\mbR^2} \bar{\partial}\tilde{\varphi}(z) R_z(H_r)\dot{H}_rR_z(H_r)\, dx dy,
\end{equation}
where $\tilde{\varphi}$ is an almost analytic extension of a test function $\varphi,$ $z=x+iy,$ $\bar{\partial} = \frac{1}{2}(\partial_x + i\partial_y),$ and $H_r$ is as in the above lemma, a $C^1$ path in $\clA.$ 
Here and below, $\dot{H}_r$ is the derivative of $H_r.$

For $H\in\clA,$ $V\in\clA_0,$ and $\varphi\in C_c(\mbR),$ the operator $E_H(\mathrm{supp}\varphi)V\varphi(H)$ belongs to the trace class, although $V\varphi(H)$ may not (yet it shares the same nonzero eigenvalues).
For any test function $\varphi,$ the infinitesimal spectral shift measure $\Phi(\varphi)$ given by the formula \eqref{F: ISSM defn} is a 1-form on the affine space $\clA,$ while for any $H\in\clA$ and perturbation $V\in\clA_0,$ $\Phi_H(V)$ is a generalised function, in particular a real-valued measure.
\begin{prop}
The 1-form $\Phi(\varphi)$ is exact for any test function $\varphi.$
\end{prop}
\begin{proof}[Sketch]
For any test function $\varphi,$ let $\theta_H^\varphi$ denote the integral of $\Phi(\varphi)$ along the line from some fixed $H_0$ to $H.$ 
The aim is to show $d\theta^\varphi_H(V) = \Phi_H(V)(\varphi).$
Beginning with the left hand side and letting $W = H - H_0$ and $H_r = H_0 + rW,$
\begin{align*}
 d\theta^\varphi_H(V) &= \frac{d}{ds}\theta^\varphi_{H+sV}\Big|_{s=0} = \lim_{s\to 0} \int_0^1 \frac{1}{s} \big( \Phi_{H_r+rsV}(W+sV)(\varphi) - \Phi_{H_r}(W)(\varphi) \big)\,dr
\\& = \lim_{s\to 0} \int_0^1 \left(\Phi_{H_r+rsV}(V)(\varphi) + \frac{1}{s}\big(\Phi_{H_r+rsV}(W)(\varphi) - \Phi_{H_r}(W)(\varphi)\big) \right)\,dr
\\& = \int_0^1 \Phi_{H_r}(V)(\varphi) \,dr + \int_0^1 \frac{d}{ds}\Phi_{H_r+rsV}(W)(\varphi)\Big|_{s=0}\,dr.
\end{align*}
The last line can be justified using the theory of double operator integrals, or \eqref{F: H-S DOI}, which can also be used to show the equalities
\begin{align*}
 \frac{d}{ds}\Phi_{H_r+rsV}(W)(\varphi)\Big|_{s=0} &= \Tr\left(\frac{d}{ds}\varphi(H_r+rsV)\Big|_{s=0}W\right)
\\&= r\Tr\left(\frac{d}{dr}\varphi(H_r)V\right)
= r\frac{d}{dr}\Phi_{H_r}(V)(\varphi).
\end{align*}
Therefore, we have
\begin{equation*}
 d\theta^\varphi_H(V) = \int_0^1 \Phi_{H_r}(V)(\varphi)\,dr + \int_0^1 r\frac{d}{dr}\Phi_{H_r}(V)(\varphi)\,dr.
\end{equation*}
Integrating the second term by parts gives the result.
\end{proof}

The {\em spectral shift measure} $\xi(\varphi) = \xi(\varphi;H_1,H_0),$ defined as the integral of $\Phi(\varphi)$ by the Birman-Solomyak formula \eqref{F: B-S formula}, therefore does not depend on the piecewise $C^1$ path from $H_0$ to $H_1.$

\begin{prop}
Let $H_r$ be a $C^1$ path in $\clA$ and let $\varphi$ be a test function. 
Then, for the path $r\mapsto\varphi(H_r),$ the chain rule holds under the trace in the sense that
\begin{equation}\label{F: chain rule}
\Tr\left(\frac{d\varphi(H_r)}{dr}f(H_r)\right) = \Tr\left(E_{H_r}(\mathrm{supp}\varphi)\dot{H}_r\varphi'(H_r)f(H_r)\right),
\end{equation}
for any bounded Borel function $f.$
\end{prop}

This proposition can be proved e.g. with the aid of \eqref{F: H-S DOI}. 
Two important corollaries follow, namely the trace formula and the invariance principle.
These are respectively obtained by integrating \eqref{F: chain rule} with $f = 1$ and $f = f\circ\varphi.$ 

\begin{cor}
The spectral shift measure satisfies the formula
\begin{equation}\label{F: trace formula}
 \Tr(\varphi(H_1) - \varphi(H_1)) = \xi(\varphi';H_1,H_0),
\end{equation}
for any $H_0,H_1\in\clA$ and test function $\varphi.$
\end{cor}
Combining \eqref{F: trace formula} with Krein's classical result \eqref{F: Kreins trace formula} implies that in the case of trace class perturbations, the spectral shift measure is absolutely continuous and its density coincides with the classical SSF.

\begin{cor}\label{C: invariance principle}
Let $H_0,H_1\in\clA,$ let $\varphi$ be a real-valued test function, and let $\xi$ and $\xi_\varphi$ be the spectral shift measures of the pairs $H_1,H_0$ and $\varphi(H_1),\varphi(H_0)$ respectively.
Then 
$
 \xi_\varphi(f) = \xi(f\circ\varphi\cdot\varphi')
$
for any bounded Borel $f.$
\end{cor}

We conclude this section with a proof of the absolute continuity of the spectral shift measure, consisting in its reduction to the trace class case.
\begin{prop}
The spectral shift measure is absolutely continuous.
\end{prop}
\begin{proof}
Let $\mu$ be the singular part of the spectral shift measure. 
We will show that $\mu$ is translation invariant.
It therefore must be a multiple of Lebesgue measure, leaving $\mu=0$ as the only possibility.
Fix $0<\varepsilon<<1$ and let $E$ be a Borel subset of $[\varepsilon,1-\varepsilon]$ with zero Lebesgue measure.
To demonstrate translation invariance it is enough to show that $\mu(E) = \mu(a+E)$ for any real $a.$
For any $a,b\in\mbR$ with $b-a>2,$ consider a test function $\varphi_{a,b}$ whose graph looks like a smoothed isosceles trapezium with height 1 which is stretched over the interval $[a,b]$ and has slopes of the sides equal to $\pm 1.$
More precisely $\varphi_{a,b}$ is subject to the constraints: $\varphi_{a,b}(\lambda) = \lambda - a$ on $[a+\varepsilon,a+1-\varepsilon],$ $\varphi_{a,b}(\lambda) = b - \lambda$ on $[b-1+\varepsilon,b-\varepsilon],$ and except as already specified $\varphi_{a,b}$ does not take values in $[\varepsilon,1-\varepsilon].$
Let $1_A$ denote the indicator of $A.$
By construction, $1_E\circ \varphi_{a,b}\cdot\varphi'_{a,b} = 1_{a+E} - 1_{b-E}.$
Then by Corollary \ref{C: invariance principle},
\begin{equation*}
\xi_{\varphi_{a,b}}(E) = \xi(a+E) - \xi(b-E) = \mu(a+E) - \mu(b-E).
\end{equation*}
The left hand side is zero, since $\xi_{\varphi_{a,b}}$ is absolutely continuous by Krein's result for trace class perturbations.
Hence choosing $b$ such that $b>2$ and $b-a>2,$ we conclude that $\mu(E) = \mu(b-E) = \mu(a+E).$
\end{proof}

\section{Singular SSF}\label{S: sSSF}
Let $H_0$ be an arbitrary self-adjoint operator on a Hilbert space $\clH$ (complex, separable, with inner product linear in the second argument) and let $F\colon\clH\to\clK$ be a rigging operator (a bounded operator with trivial kernel and cokernel) which is {\em $H_0$-Hilbert-Schmidt,} by which we mean that the operator $FR_z(H_0)$ is Hilbert-Schmidt for some, hence any, nonreal $z.$
Put $\clA_0(F) := F^*\clB_{sa}(\clK)F,$ where $\clB_{sa}(\clK)$ denotes the bounded self-adjoint operators on $\clK,$ and equip it with the norm $\|F^*JF\|_F := \|J\|$ making it a real Banach space isomorphic to $\clB_{sa}(\clK).$
Consider the affine space $\clA(F) := H_0 + \clA_0(F).$

We now list some properties which follow from the definition of $\clA(F).$ 
By the Kato-Rellich theorem all operators from $\clA(F)$ are self-adjoint and share a common domain. 
By Weyl's theorem all operators share a common essential spectrum $\sigma_{ess}.$
The rigging operator $F$ is $H$-Hilbert-Schmidt for any $H\in\clA(F)$ and any two operators from $\clA(F)$ are resolvent comparable.
For each $H\in\clA(F)$ the set $\Lambda(H,F;\clL_1)$ is a full set, in particular the Limiting Absorption Principle holds. 
For any $V=F^*JF \in\clA_0(F),$ any $H\in\clA(F),$ and any two Borel functions $\varphi,\psi$ dominated by $\mathrm{const.}(1+x^2)^{-1/2},$ the operators $JF(\varphi \psi)(H)F^*$ and $\psi(H)V\varphi(H)$ belong to the trace class, having equal traces.
The function $\clA(F)\ni H \mapsto FR_z(H)\in \clL_2(\clH,\clK)$ is Lipschitz continuous.

From Section \ref{S: SSF}, the infinitesimal spectral shift measure $\Phi$ given by \eqref{F: ISSM defn} is exact on $\clA(F).$
Moreover, its integral \eqref{F: B-S formula} over any piecewise $C^1$ path from $H_0$ to $H_1$ is an absolutely continuous measure whose density is the SSF $\xi(\lambda;H_1,H_0).$
In this setting, the SSF belongs to $L_1(\mbR,(1+x^2)^{-1}dx).$ 
Its Poisson integral, which we will call the {\em smoothed SSF,} is given by
\begin{align}
 \xi(z;H_1,H_0) &:= \xi(\pi^{-1}\Im R_z;H_1,H_0) \nonumber
\\&= \frac{1}{\pi}\int_0^1 y\Tr\left(R_{\bar{z}}(H_r)\dot{H}_rR_z(H_r)\right) \,dr, \label{F: smoothed SSF}
\end{align}
where $y = \Im z$ and $H_r$ is a piecewise $C^1$ path in $\clA(F).$
By a well known result on the convergence of Poisson integrals (see e.g. \cite[Theorem 2.5.4]{Sim}), 
\begin{equation}\label{F: smoothed SSF converges to SSF}
\xi(\lambda;H_1,H_0) = \lim_{y\to 0^+}\xi(\lambda+iy;H_1,H_0) \quad \text{ for a.e. } \lambda\in\mbR.
\end{equation}

We now consider the Lebesgue decomposition $\Phi = \Phi^{(a)} + \Phi^{(s)}$ of the infinitesimal spectral shift measure.
For $H\in\clA(F)$ and $V\in\clA_0(F),$ the absolutely continuous and singular parts are respectively given by replacing $(\cdot)$ with $(a)$ and $(s)$ in the formula
\[
 \Phi^{(\cdot)}_H(V) = \Tr\left(E^{(\cdot)}_H(\mathrm{supp}\varphi)V\varphi(H)\right), \quad \varphi\in C_c(\mbR).
\]
\begin{lemma}\label{L: a.c.ISSM}
The absolutely continuous part of the infinitesimal spectral shift measure has the representation:
\begin{equation}\label{F: a.c.ISSM}
 \Phi^{(a)}_H(V)(\varphi) = \frac{1}{\pi}\int_\mbR\varphi(\lambda)\lim_{y\to 0^+}y\Tr(R_{\lambda-iy}(H)VR_{\lambda+iy}(H))\,d\lambda, \quad \phi\in C_c(\mbR).
\end{equation}
\end{lemma}
\begin{proof}
Let $V = F^*JF.$ 
The trace in the integrand on the right of \eqref{F: a.c.ISSM}, being equal to $\Tr(J\Im T_{\lambda+iy}(H))$ since $\Im R_{\lambda+iy} = y R_{\lambda+iy}R_{\lambda-iy},$ converges to $\Tr(J\Im T_{\lambda+i0}(H))$ for all $\lambda$ from the full set $\Lambda(H,F;\clL_1).$
Without loss of generality both $J$ and $\varphi$ can be assumed to be positive, in which case the integrand is also positive.
Choosing an orthonormal basis $\{\psi_j\}$ of the auxiliary Hilbert space $\clK,$ the right hand side of \eqref{F: a.c.ISSM} is equal to
\[
\frac{1}{\pi}\int_\mbR\sum_{j=1}^\infty \varphi(\lambda)\scal{\psi_j,J\Im T_{\lambda+i0}(H)\psi_j}\,d\lambda.
\]
Interchanging the sum and integral, we consider the terms of the resulting sum:
\[
 \int_{\mbR}\varphi(\lambda)\lim_{y\to 0^+}\frac{1}{\pi}\scal{F^*J\psi_j,\Im R_{\lambda+iy}(H)F^*\psi_j}\,d\lambda.
\]
The inner product here is the Poisson integral of the measure $\scal{F^*J\psi_j,E F^*\psi_j},$ where $E$ is the spectral measure of $H,$ and therefore converges a.e. to the density of its absolutely continuous part. 
Thus the terms of the sum can be rewritten as
$
\langle F^*J\psi_j,$ $\varphi(H)P^{(a)}(H)F^*\psi_j \rangle,$
where $P^{(a)}(H)$ is the projection onto the absolutely continuous subspace of $H.$ For any $\varphi$ from $C_c(\mbR)$ the operator $JF\varphi(H)P^{(a)}(H)F^*$ belongs to the trace class and summing over $j$ gives the left hand side of \eqref{F: a.c.ISSM}.
\end{proof}

\begin{thm}\label{T: a.c.SSF}
For $\varphi\in C_c(\mbR),$ the integral $\xi^{(a)}(\varphi;\{H_r\}):= \eqref{F: a.c.SSM}$ of $\Phi^{(a)}(\varphi)$ over a piecewise analytic path $H_r$ in $\clA(F),$ defines an absolutely continuous measure whose density $\xia\left(\lambda;\{H_r\}\right),$ the absolutely continuous SSF, for a.e. $\lambda\in\mbR$ satisfies the equality
\begin{equation}\label{F: a.c.SSF}
\xia\left(\lambda;\{H_r\}\right) = \frac{1}{\pi}\int_0^1 \lim_{y\to 0^+}y\Tr\left(R_{\lambda-iy}(H_r)\dot{H}_r R_{\lambda+iy}(H_r)\right)\,dr.
\end{equation}
\end{thm}
\begin{proof}
Without loss of generality we can assume $H_r$ is analytic, so let $J_r$ be an analytic path in $\clB_{sa}(\clK)$ such that $H_r = H_0 + F^*J_rF.$
We consider the function defined by
\begin{equation}\label{F: Phia integrand}
 (\lambda,r) \mapsto \frac{1}{\pi}\Tr\left(\dot{J}_r\Im T_{\lambda+i0}(H_r)\right).
\end{equation}
This function is defined on the set $\Gamma := \{(\lambda,r) : \lambda\in\Lambda(H_r,F;\clL_1)\},$ which is a full set in the plane.
For any fixed $\lambda$ from the full set $\Lambda(\{H_r\},F;\clL_1),$ the pair $(\lambda,r)$ belongs to $\Gamma$ if and only if $r$ does not belong to the discrete resonance set $R(\lambda;\{H_r\}).$
For each such $r$ and thus a.e. $r\in[0,1]$ the integrand on the right hand side of \eqref{F: a.c.SSF} is equal to the value of \eqref{F: Phia integrand}.
We will now check that \eqref{F: Phia integrand} is integrable over a bounded rectangle $\Delta\times[0,1].$ 
For this it suffices to show that as a function of $\lambda$ on the bounded interval $\Delta$ it has locally bounded $L_1$-norm with respect to $r.$
\begin{align*}
 \frac{1}{\pi}\int_{\Delta}\left|\Tr\left(\dot{J}_r\Im T_{\lambda+i0}(H_r)\right)\right|\,d\lambda &\leq \frac{1}{\pi}\int_{\Delta}\Tr\left(\big|\dot{J}_r\big|\Im T_{\lambda+i0}(H_r)\right)\,d\lambda
\\ &= \Phi^{(a)}_{H_r}(F^*|\dot{J_r}|F)(\Delta)
\\ &\leq \big\|\dot{J}_r\big\|\big\|FE_{H_r}(\Delta)\big\|_2^2,
\end{align*}
where in the second line we have used Lemma \ref{L: a.c.ISSM} and in the third $E_{H_r}$ denotes the spectral measure of $H_r.$
It follows from our assumptions that the last expression is a locally bounded function of $r.$ 
Therefore, the following use of Fubini's theorem is justified. 
For any $\varphi\in C_c(\mbR),$ we have 
\begin{multline*}
\int_\mbR \varphi(\lambda)\int_0^1 \frac{1}{\pi}\Tr\left(\dot{J}_r\Im T_{\lambda+i0}(H_r)\right) dr\,d\lambda \\= \int_0^1 \int_\mbR \varphi(\lambda)\frac{1}{\pi}\Tr\left(\dot{J}_r\Im T_{\lambda+i0}(H_r)\right)d\lambda\,dr.
\end{multline*}
By Lemma \ref{L: a.c.ISSM} the right hand side is equal to
$
\int_0^1 \Phi^{(a)}_{H_r}(\dot{H}_r)(\varphi)\,dr = \xi^{(a)}(\varphi;\{H_r\}).
$
It follows that $\xi^{(a)}$ is absolutely continuous with density $\xi(\lambda;\{H_r\})$ equal a.e. to
$
\pi^{-1}\int_0^1 \Tr\big(\dot{J}_r\Im T_{\lambda+i0}(H_r)\big) dr,
$
which coincides with \eqref{F: Phia integrand} for any $\lambda$ from the full set $\Lambda(\{H_r\},F;\clL_1).$
\end{proof}

\begin{cor}\label{C: sSSF}
The integral \eqref{F: xis defn} of $\Phi^{(s)}$ over a piecewise analytic path $H_r$ in $\clA(F)$ defines an absolutely continuous measure whose density, the singular SSF, is the difference of the SSF and the absolutely continuous SSF:
\begin{equation*}
\xis(\lambda;\{H_r\}) = \xi(\lambda;H_1,H_0) - \xia(\lambda;\{H_r\}).
\end{equation*}
\end{cor}

\section{Singular SSF and total resonance index}
In this section we will restrict our attention to a straight line $H_r = H_0 +rV,$ with $V = F^*JF,$ connecting two self-adjoint operators $H_0$ and $H_1$ in the affine space $\clA(F).$
For straight paths we usually modify notation which indicates dependence on a piecewise analytic path $H_r$ by instead writing the pair of endpoints $H_1,H_0,$ or a point and direction $H_0,V$ if the endpoints are not important, e.g. $\xi^{(s)}(\lambda;\{H_r\}) \leftrightarrow \xi^{(s)}(\lambda;H_1,H_0)$ and $R(\lambda;\{H_r\}) \leftrightarrow R(\lambda;H_0,V).$

For a fixed essentially regular point $\lambda$ of the path $H_r = H_0 + rV,$ let $r_\lambda\in R(\lambda;H_0,V)$ be a fixed resonance point, i.e. $r_\lambda$ is a real pole of the meromorphic function $r\mapsto T_{\lambda+i0}(H_r).$
By shifting the point $\lambda$ slightly to $\lambda+iy$ for small positive (or negative) $y,$ the pole $r_\lambda$ may in general split into finitely many poles $r^1_{\lambda+iy},\ldots,r^N_{\lambda+iy}$ of $r\mapsto T_{\lambda+iy}(H_r).$
These $N$ poles (counted with multiplicities), called resonance points of $\lambda+iy$ and collectively called the $r_\lambda$-group, are stable for small $y$ and none of them can be real.
Changing the sign of $y$ results in a reflection of the $r_\lambda$-group about the real axis and resonance points of $\bar{z}$ are called anti-resonance points of $z.$
For small positive $y,$ let $N_+$ be the number of resonance points of $\lambda+iy$ from the $r_\lambda$-group which lie in the upper half-plane $\mbC_+.$ 
Similarly let $N_-$ be the number of resonance points in the lower half-plane $\mbC_-,$ or equivalently the number of anti-resonance points in $\mbC_+.$
The {\em resonance index} $\mathrm{ind}_{res}(\lambda;H_{r_\lambda},V)$ is by definition the difference $N_+ - N_-.$

If $\lambda$ is outside of the essential spectrum $\sigma_{ess},$ $\mathrm{ind}_{res}(\lambda;H_{r_\lambda},V)$ counts the net number of eigenvalues of the path $H_r$ which cross the point $\lambda$ in the positive direction as $r$ crosses $r_\lambda,$ 
so the resonance index can be considered as an extension of this concept of infinitesimal spectral flow into $\sigma_{ess};$
see \cite{Az9} for more information.

If $\lambda\in\Lambda(H_s,F)$ (which must hold for some real $s$ by the assumption that $\lambda$ is essentially regular), then $T_{z}(H_r) = T_{z}(H_s)(1 + (r-s)JT_{z}(H_s))^{-1}$ for $z=\lambda+iy,$ $y\geq 0,$ and resonance points $r_{z}$ correspond to eigenvalues $\sigma_z(s) = (s-r_z)^{-1}$ of the compact operator $JT_{z}(H_s).$
The eigenvalue $\sigma_z(s)$ and the resonance point $r_z$ lie in the same half-plane so that $N_\pm$ is the number of eigenvalues $\sigma_{\lambda\pm iy}(s)$ in $\mbC_+,$ counting multiplicities, which converge to $\sigma_\lambda(s)$ as $y\to 0^+.$
A simple argument (see \cite[Proposition 3.2.3]{Az9}) shows that the Riesz projection onto the eigenspace of $JT_z(H_s)$ corresponding to the eigenvalue $\sigma_z(s)$ does not depend on $s$ and coincides with the residue of $r\mapsto JT_z(H_r)$ at $r_\lambda:$ 
\[
 \frac{1}{2\pi i}\oint_{C(\sigma_z(s))} (\sigma - JT_z(H_s))^{-1} \,d\sigma = \frac{1}{2\pi i}\oint_{C(r_z)} JT_z(H_r) \,dr.
\]
It follows that for small enough $y,$ we have
\[
N_\pm = \Tr\left(\frac{1}{2\pi i} \oint_{C_+(r_\lambda)}JT_{\lambda\pm iy}(H_r) \,dr \right),
\] 
where $C_+(r_\lambda)$ is the upper semicircle of a small circle around $r_\lambda$ which for $0<y<<1$ encloses only the resonance points of the $r_\lambda$-group (see the picture below).
Hence the resonance index has the representation:
\begin{equation}\label{F: res ind}
 \mathrm{ind}_{res}(\lambda;H_{r_\lambda},V) = \Tr\left(\frac{1}{\pi} \oint_{C_+(r_\lambda)} J\Im T_{\lambda+iy}(H_r)\,dr\right).
\end{equation}
In the current context of $\clA(F),$ the function $r\mapsto \pi^{-1}J\Im T_{\lambda+iy}(H_r)$ is a meromorphic trace class valued function so that the trace and integral in \eqref{F: res ind} can be interchanged.

Let $[a,b]$ be an interval containing $r_\lambda$ and no other resonance points of the path $H_r$ and let $L$ be a contour in $\mbC$ from $a$ to $b$ which for small $y$ circumvents resonance and anti-resonance points of the $r_\lambda$-group in $\mbC_+,$ as shown in the picture below.

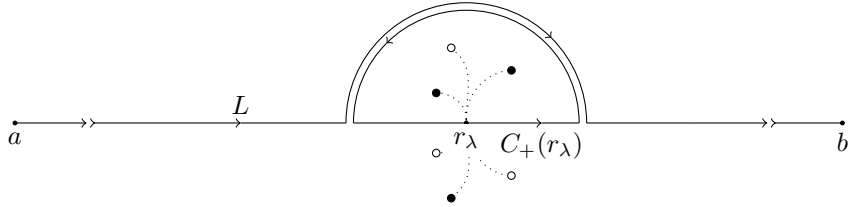
\begin{center}
\begin{tikzpicture}
\draw[fill] (4.6,0.4) circle (0.05);
\draw[dotted] (5,0.05) to[out=90,in=-10] (4.65,0.4);
\draw[fill] (5.6,0.7) circle (0.05);
\draw[dotted] (5,0.05) to[out=90,in=190] (5.55,0.7);
\draw (4.8,1) circle (0.05);
\draw[dotted] (5,0.05) to[out=90,in=-40] (4.85,.95);

\draw (4.6,-0.4) circle (0.05);
\draw[dotted] (5,0.05) to[out=-90,in=10] (4.65,-0.4);
\draw (5.6,-0.7) circle (0.05);
\draw[dotted] (5,0.05) to[out=-90,in=-190] (5.55,-0.7);
\draw[fill] (4.8,-1) circle (0.05);
\draw[dotted] (5,0.05) to[out=-90,in=40] (4.85,-.95);
	
\node [above] at (2,0) {$L$};
\draw[->] (-1,0) -- (-0.05,0);
\draw[>->] (0,0) -- (2,0);
\draw (2,0) -- (3.4,0);
\draw[->] (3.4,0) arc (180:45:1.6);
\draw (6.6,0) arc (0:45:1.6);
\draw (6.6,0) -- (8,0);
\draw[->] (8,0) -- (9,0);
\draw[>-] (9.05,0) -- (10,0);

\node [below] at (6,0) {$C_+(r_\lambda)$};
\draw[->] (3.5,0) -- (6,0);
\draw (6,0) -- (6.5,0);
\draw[->] (6.5,0) arc (0:135:1.5);
\draw (3.5,0) arc (180:135:1.5);

\draw[fill] (-1,0) circle (0.025);
\node [below] at (-1,0) {$a$};
\draw[fill] (10,0) circle (0.025);
\node [below] at (10,0) {$b$};
\draw[fill] (5,0) circle (0.025);
\node [below,fill=white,opacity=1] at (5,0) {$r_\lambda$};
\end{tikzpicture}
\end{center}
The semicircle in this picture may be tiny in comparison to the interval $[a,b].$
Note that the resonance point $r_\lambda$ is a pole of the meromorphic function $r\mapsto \pi^{-1}J\Im T_{\lambda+i0}(H_r),$ which has no other poles in the vicinity of $C_+(r_\lambda).$

The integral over $[a,b]$ of the meromorphic function $r \mapsto \pi^{-1}J\Im T_{\lambda+iy}(H_r),$ $0<y<<1,$ can be decomposed into the sum of its integrals over $L$ and $C_+(r_\lambda).$
After taking the trace and using \eqref{F: res ind}, for small $y$ we obtain
\begin{multline}\label{F: L2 = L1 + P}
\frac{1}{\pi}\int_a^b \Tr( J\Im T_{\lambda+iy}(H_r) ) \,dr 
\\= \frac{1}{\pi}\int_L \Tr( J\Im T_{\lambda+iy}(H_r) ) \,dr + \mathrm{ind}_{res}(\lambda;H_{r_\lambda},V).
\end{multline}
The left hand side of \eqref{F: L2 = L1 + P} is the smoothed SSF $\xi(\lambda+iy;H_b,H_a)$ which by \eqref{F: smoothed SSF converges to SSF} converges for a.e. $\lambda$ to the SSF $\xi(\lambda;H_b,H_a)$ as $y\to 0^+.$
Whereas we will show that the first term on the right hand side of \eqref{F: L2 = L1 + P} converges a.e. to the absolutely continuous SSF $\xi^{(a)}(\lambda;H_b,H_a).$
Firstly, if $\lambda\in\Lambda(\{H_r\},F;\clL_1),$
\[
 \lim_{y\to 0^+}\frac{1}{\pi}\int_L \Tr( J\Im T_{\lambda+iy}(H_r) ) \,dr = \frac{1}{\pi}\int_L \Tr( J\Im T_{\lambda+i0}(H_r) ) \,dr,
\]
since on the contour $L$ the integrand on the left converges uniformly to the integrand on the right as $y\to 0^+.$
Secondly, in the last integral we can replace the contour $L$ by the interval $[a,b],$ because the function $r\mapsto \pi^{-1}\Tr(J\Im T_{\lambda+i0}(H_r))$ has no poles within $C_+(r_\lambda)$ and moreover it admits analytic continuation to the real axis --
a fact that we shall delay proving until it conveniently follows from Theorem \ref{T: S'(r)} below.
By \eqref{F: a.c.SSF}, the result for a.e. $\lambda$ is $\xi^{(a)}(\lambda;H_b,H_a).$

Therefore, for a.e. $\lambda$ from the full set $\Lambda(\{H_r\},F;\clL_1)$ we have
\[
 \xi(\lambda;H_a,H_b) = \xi^{(a)}(\lambda;H_a,H_b) + \mathrm{ind}_{res}(\lambda;H_{r_\lambda},V),
\]
for any interval $[a,b]$ containing a single resonance point $r_\lambda.$
Using additivity of the singular SSF along the path $H_r$ the above argument proves
\begin{thm}\label{T: sSSF = res ind}
Let $H_0$ and $H_1$ be two self-adjoint operators from $\clA(F)$ and let $H_r = H_0 +rV.$ 
For a.e. $\lambda$ from the full set $\Lambda(H_0,F;\clL_1)\cap\Lambda(H_1,F;\clL_1),$
\[
 \xi^{(s)}(\lambda;H_1,H_0) = \sum_{r_\lambda \in [0,1]}\mathrm{ind}_{res}(\lambda;H_{r_\lambda},V),
\]
where the sum, namely the total resonance index, is taken over the finite number of resonance points $r_\lambda$ from the interval $[0,1].$
\end{thm}

\section{Singular SSF and singular \texorpdfstring{$\mu$}{mu}-invariant}\label{S: sSSF and mu inv}
A preliminary aim of this section is to establish the ordered exponential representation of the scattering matrix:
\begin{equation}\label{F: S = Texp}
S(\lambda;H_1,H_0) = \Texp\left(-2 i \int_0^1 w_+(\lambda;H_0,H_r)(\ldots)w_+(\lambda;H_r,H_0)\, dr \right),
\end{equation}
where the ingredients of this formula are: $H_r = H_0 + F^*J_rF$ is a piecewise analytic path in the affine space $\clA(F),$ 
the brackets $(\ldots)$ stand for the operator $\sqrt{\Im T_{\lambda+i0}(H_r)}\dot{J}_r\sqrt{\Im T_{\lambda+i0}(H_r)},$
and $w_+(\lambda;H_0,H_r)\colon\hlambda(H_r)\to\hlambda(H_0)$ is a wave matrix, where $\hlambda(H_r)$ is the fibre of a direct integral of Hilbert spaces $\int^{\oplus}\hlambda(H_r)\,d\lambda$ which is isomorphic to the absolutely continuous subspace $\clH^{(a)}(H_r)$ and on which the absolutely continuous part of $H_r$ acts as multiplication by $\lambda.$
The ordered exponential \eqref{F: S = Texp} is the unique solution to the ordinary differential equation
\begin{equation}\label{F: S ODE}
\frac{d}{dr}S(\lambda;H_r,H_0) = -2 i w_+(\lambda;H_0,H_r)(\ldots)w_+(\lambda;H_r,H_0)S(\lambda;H_r,H_0),
\end{equation}
where again $(\ldots)=\sqrt{\Im T_{\lambda+i0}(H_r)}\dot{J}_r\!\sqrt{\Im T_{\lambda+i0}(H_r)}.$
Further information about the ordered exponential can be found in the appendix to \cite{Az3v6}.

Proving \eqref{F: S ODE} requires objects of stationary scattering theory such as the scattering matrix $S(\lambda;H_r,H_0),$ for fixed $\lambda,$ to be well defined for a continuous family of operators $H_r.$ 
Classical approaches (e.g. \cite{Ya,BE}) define such objects via arbitrary direct integral decompositions, so that we cannot ask questions about $S(\lambda;H_r,H_0)$ for a fixed $\lambda.$
In \cite{KK}, T. Kato and S.\,T. Kuroda construct the wave matrices $w_\pm(\lambda;H_r,H_0)$ for a full set of values $\lambda,$ but it remains unclear how to investigate their dependence on $r.$
An approach to stationary scattering theory which allows a proof of \eqref{F: S = Texp} is given in \cite{Az3v6}, for the case that the perturbation $H_1-H_0$ is trace class.
It is constructive in the sense that the wave matrices $w_\pm(\lambda;H_r,H_0),$ scattering matrix $S(\lambda;H_0,H_r),$ and other such objects are introduced by explicit formulas for every value of the spectral parameter $\lambda$ from a pre-defined full set, namely the set $\Lambda(H_0,F)\cap\Lambda(H_r,F).$ 
In the following subsection, this approach is generalised and simplified.
Further discussion about this approach can be found in the introduction to \cite{Az9}.

\subsection{Stationary scattering theory}\label{S: SST}
Let $H$ be a self-adjoint operator on $\clH$ and let $F\colon\clH\to\clK$ be an arbitrary rigging operator;
for this subsection we no longer require $F$ to be $H$-Hilbert-Schmidt.
Because the rigging operator will be fixed, dependence on it is often omitted from notation.

For $y>0,$ the range of the operator $\sqrt{\Im T_{\lambda+iy}(H)}$ is dense in the auxiliary Hilbert space $\clK.$
However for $\lambda\in\Lambda(H,F),$ the range of $\sqrt{\Im T_{\lambda+i0}(H)}$ may no longer be dense and we denote its closure by $\hlambda(H)=\hlambda(H,F).$
The field of fibre Hilbert spaces $\{\hlambda(H) : \lambda\in\Lambda(H,F)\}$ is measurable, in the sense that the orthogonal projections onto $\hlambda$ are weakly measurable, and hence defines a direct integral of Hilbert spaces 
\begin{equation*}
	\euH(H) := \int^\oplus_{\Lambda(H,F)}\hlambda(H)\,d\lambda.
\end{equation*}
This Hilbert space is the closed subspace of $L_2(\Lambda(H,F),\clK)$ consisting of those functions $f\colon \Lambda(H,F) \to \clK$ such that $f(\lambda)\in\hlambda(H)$ for a.e. $\lambda\in\Lambda(H,F).$

For any $\lambda\in\Lambda(H,F),$ let the {\em evaluation operator} $\euE_\lambda(H) = \euE_\lambda(H,F)$ be defined on the range of $F^*$ by 
\begin{equation}\label{F: eval opr defn}
 \euE_\lambda(H) = \sqrt{\pi^{-1}\Im T_{\lambda+i0}(H)}(F^*)^{-1}.
\end{equation}
Then for any $f$ from $\rng F^*,$ the function $\lambda\mapsto \euE_\lambda(H)f$ belongs to $\euH(H).$
Moreover, the measurable family of operators $\{\euE_\lambda(H) : \lambda\in\Lambda(H,F)\}$ defines a bounded operator $\euE(H) = \euE(H,F)$ from $\rng F^*$ to $\euH(H)$ with norm $\leq 1,$ which by the density of the range of $F^*$ in $\clH$ extends to a bounded operator
$
\euE(H)\colon \clH \to \euH(H)
$
whose norm is $\leq 1.$
Indeed,
\begin{align}
 \|\euE(H,F)f\|_{\euH(H)}^2 & = \frac{1}{\pi}\int_{\Lambda(H,F)}\lim_{y\to 0^+}\scal{f,\Im R_{\lambda+iy}(H)f}\,d\lambda \nonumber
\\ & = \|E(\Lambda(H,F))f\|^2. \label{F: Ef integrable}
\end{align}
For $f\in\rng F^*,$ the first equality here is a consequence of the definition \eqref{F: eval opr defn}, while the second, in which $E$ denotes the spectral measure of $H,$ holds due to properties of the Poisson integral (see e.g. \cite[Theorem 2.5.4]{Sim}): the measure $\scal{f,Ef}$ is purely absolutely continuous on the set $\Lambda(H,F),$ since its Poisson integral $\pi^{-1}\scal{f,\Im R_{\lambda+iy}(H)f}$ has finite limits there, and its density is a.e. equal to the limit of its Poisson integral.

The spectral measure $E$ of $H$ is purely absolutely continuous on set $\Lambda(H,F).$ 
If this were not true, there would exist a null set $X\subset\Lambda(H,F)$ for which $E(X)\neq 0$ and hence $\scal{f,E(X)f}\neq 0$ for some $f\in\rng F^*,$ which since $\scal{f,Ef}$ is absolutely continuous on $\Lambda(H,F)$ is impossible.
Therefore the second equality of \eqref{F: Ef integrable} holds for any $f\in\clH.$
It follows that the first equality also holds for any $f\in\clH.$

\begin{thm}\label{T: spectral thm}
Let $H$ be a self-adjoint operator on a Hilbert space $\clH,$ let $E$ be its spectral measure, and let $F\colon\clH\to\clK$ be a rigging operator. 
The operator $\euE(H)\colon\clH\to\euH(H)$ defined above is a partial isometry with initial space $E(\Lambda(H,F))\clH$ and final space $\euH(H).$
Moreover, $\euE(H)$ diagonalises the operator $E(\Lambda(H,F))H$ in the sense that for all $f\in\dom H$
\begin{equation}\label{F: H diagonalised}
 (\euE(H)Hf)(\lambda) = \lambda(\euE(H)f)(\lambda) \quad \forall \text{ a.e. } \lambda\in\Lambda(H,F)
\end{equation}
and if $h$ is a bounded Borel function whose minimal support is a subset of $\Lambda(H,F),$ then for all $f\in\clH$
\begin{equation}\label{F: h(H) diagonalised}
 (\euE(H)h(H)f)(\lambda) = h(\lambda)(\euE(H)f)(\lambda) \quad \forall \text{ a.e. } \lambda\in\Lambda(H,F).
\end{equation}
\end{thm}

Here, $(\euE(H)f)(\lambda)$ denotes the value of a function from the equivalence class $\euE(H)f$ of functions equal a.e. 
If $f\in\rng F^*,$ then a representative from this equivalence class is determined by the evaluation operator: $\lambda\mapsto\euE_\lambda(H)f.$

If the Limiting Absorption Principle holds for $H,$ then $\mbR\setminus\Lambda(H,F)$ is a core of the singular spectrum of $H$ and $E(\Lambda(H,F))$ is equal to the projection onto the absolutely continuous subspace $\clH^{(a)}(H).$
In this case $\euE(H)$ diagonalises the absolutely continuous part of $H.$

\begin{proof}
The equality \eqref{F: Ef integrable} implies that $\euE(H)$ is a partial isometry with initial space $E(\Lambda(H,F))\clH.$

We will now show, for any Borel subset $\Delta\subset\Lambda(H,F),$ that if $E(\Delta)f = 0$ then $(\euE(H)f)(\lambda) = 0$ for a.e. $\lambda\in\Delta.$
Supposing $E(\Delta)f = 0,$ let $f_n$ be a sequence from the range of $F^*$ converging to $f.$
Then 
\begin{align*}
\int_{\Delta}\|(\euE(H)f)(\lambda) - \euE_\lambda(H)f_n\|^2\,d\lambda &= \int_{\Delta}\|(\euE(H)(f-f_n))(\lambda)\|^2 \,d\lambda
\\&\leq \int_{\Lambda(H,F)}\|(\euE(H)(f-f_n))(\lambda)\|^2 \,d\lambda
\\&= \|\euE(H)(f-f_n)\|^2 
\\&\leq \|f-f_n\|^2 \to 0.
\end{align*}
Moreover, since $f_n\in\rng F^*,$ for any $n$ we have 
\[
\int_{\Delta}\|\euE_\lambda(H)f_n\|^2\,d\lambda = \frac{1}{\pi}\int_\Delta \lim_{y\to 0^+}\scal{f_n,\Im R_{\lambda+iy}f_n}\,d\lambda = \|E(\Delta)f_n\|^2.
\]
Therefore, since $E(\Delta)f = 0,$ in the limit the above equality becomes
\[
 \int_{\Delta}\|(\euE(H)f)(\lambda)\|^2\,d\lambda = 0,
\]
which implies that $(\euE(H)f)(\lambda) = 0$ for a.e. $\lambda\in\Delta.$

Let $\Delta$ be a Borel subset of $\Lambda(H,F)$ and let $f\in\clH.$ 
From above it follows that $(\euE(H)E(\Delta)f)(\lambda) = 0$ for a.e. $\lambda\notin\Delta$ 
and also that $(\euE(H)E(\Delta)f)(\lambda) = (\euE(H)f)(\lambda)$ for a.e. $\lambda\in\Delta.$ 
Therefore, the equality
$
 (\euE(H) E(\Delta) f)(\lambda) = \Delta(\lambda)(\euE(H) f)(\lambda)
$ 
holds for a.e. $\lambda \in \Lambda(H,F),$ where $\Delta(\lambda)$ denotes the indicator of $\Delta.$ 
This equality implies that \eqref{F: h(H) diagonalised} holds for step functions $h(\lambda),$
hence by continuity it holds for all bounded Borel functions $h.$
And \eqref{F: h(H) diagonalised} implies \eqref{F: H diagonalised}.

It remains to show that $\euH(H)$ is the final space, for which it is enough to show that the range of $\euE(H)$ is dense in $\euH(H).$
Let $g(\lambda)$ be an element of $\euH(H)$ which is orthogonal to the range of $\euE(H).$
The equality \eqref{F: h(H) diagonalised} implies that if the range of $\euE(H)$ contains a function $f(\lambda),$ then it also contains all functions of the form $h(\lambda)f(\lambda),$ where $h$ is a scalar-valued bounded Borel function.
Hence for any $f\in\clH,$ $g(\lambda)$ must be orthogonal to $h(\lambda)(\euE(H)f)(\lambda)$ for any bounded Borel $h.$
One infers that $g(\lambda) \perp (\euE(H)f)(\lambda)$ in $\hlambda(H)$ for a.e. $\lambda \in \Lambda(H,F).$
By considering $f=F^*\varphi$ from the range of $F^*,$ in which case $(\euE(H)F^*\varphi)(\lambda) = \sqrt{\Im T_{\lambda+i0}(H)}\varphi$ for a.e. $\lambda,$ it can be concluded that $g(\lambda)$ must be orthogonal to the whole fibre Hilbert space $\hlambda(H)$ and thus $g(\lambda) = 0$ for a.e. $\lambda \in \Lambda(H,F).$
\end{proof}

\smallskip
For any nonreal $z,$ consider the polar decomposition:
\begin{equation*}
 R_z(H)F^* = U_z(H,F)|R_z(H)F^*|.
\end{equation*}
Since $R_z(H)F^*$ has trivial kernel and cokernel, the operator $U_z(H,F)\colon\clK\to\clH$ is unitary. 
For any two self-adjoint operators $H_0$ and $H_1$ and for any nonreal $z,$ the operator $w(z;H_1,H_0)$ on the auxiliary Hilbert space $\clK$ defined by 
\begin{equation*}
 w(z;H_1,H_0) = U^*_z(H_1,F)U_z(H_0,F)
\end{equation*}
will be called an {\em off-axis wave matrix.}
Obviously the off-axis wave matrices are unitary, admit the multiplicative property $w(z;H_2,H_0) = w(z;H_2,H_1)w(z;H_1,H_0),$ and satisfy the equalities $w(z;H_0,H_0) = 1$ and $w^*(z;H_1,H_0) = w(z;H_0,H_1).$

Suppose $H_1 - H_0 = F^*JF$ for $J\in\clB_{sa}(\clK).$ Then the following equalities, in which $z=\lambda\pm iy,$ $y>0,$ can be obtained from the resolvent identities.
\begin{gather}
\sqrt{\Im T_{\lambda+iy}(H_1)} w(z;H_1,H_0)\sqrt{\Im T_{\lambda+iy}(H_0)} 
 = y FR_{\bar{z}}(H_1)R_z(H_0)F^* \label{F: a(z)}
\\ \sqrt{\Im T_{\lambda+iy}(H_1)}w(z;H_1,H_0) = (1 - T_{\bar{z}}(H_1)J)\sqrt{\Im T_{\lambda+iy}(H_0)} \label{F: Ew}
\\ w(z;H_1,H_0)\sqrt{\Im T_{\lambda+iy}(H_0)} = \sqrt{\Im T_{\lambda+iy}(H_1)}(1 + JT_z(H_0)) \label{F: wE}
\end{gather}
For example,
\begin{align*}
 w(z;H_1,H_0)\sqrt{\Im T_{\lambda+iy}(H_0)} &= U^*_z(H_1)U_z(H_0)\sqrt{yFR_{\bar{z}}(H_0)R_z(H_0)F^*}
\\&= \sqrt{y}\, U^*_z(H_1)R_z(H_0)F^*
\\&= \sqrt{y}\, U^*_z(H_1)R_z(H_1)F^*(1 + JT_z(H_0))
\\&= \sqrt{\Im T_{\lambda+iy}(H_1)}(1 + JT_z(H_0)).
\end{align*}

If $\lambda$ belongs to the intersection $\Lambda(H_0,F)\cap\Lambda(H_1,F)$ then each of the right hand sides of \eqref{F: wE}, \eqref{F: Ew}, and \eqref{F: a(z)} exists as a bounded operator on $\clK$ in the limit as $y\to 0^+.$ 
For convenience, with $z = \lambda \pm iy,$ $y>0,$ we put 
\[
 \euE_{z}(H) := \sqrt{\pi^{-1}\Im T_{\lambda+iy}(H)}(F^*)^{-1}.
\]
Then for any $\lambda\in\Lambda(H,F)$ and any $f\in\rng F^*,$ $\euE_z(H)f$ converges to $\euE_\lambda(H)f$ as $y\to 0^+.$
For $\Im z \neq 0,$ the density of $\rng \euE_{z}(H)$ in $\clK$ and \eqref{F: a(z)} imply that the off-axis wave matrices are determined by the numbers
\[
 \scal{\euE_{z}(H_1)f,w(z;H_1,H_0)\euE_{z}(H_0)g} = \frac{y}{\pi}\scal{R_{z}(H_1)f,R_{z}(H_0)g}, \quad \forall f,g\in\rng F^*.
\] 
\begin{prop}\label{P: wave matrix}
Suppose that $H_0$ and $H_1$ are self-adjoint operators such that $H_1 - H_0 \in F^*\clB_{sa}(\clK)F.$ 
Then for any $\lambda\in\Lambda(H_0,F)\cap\Lambda(H_1,F),$ there exist bounded operators
\begin{equation}\label{F: wave matrix}
 w_\pm(\lambda;H_1,H_0)\colon \hlambda(H_0,F) \to \hlambda(H_1,F),
\end{equation}
namely the wave matrices, which have norm $\leq 1$ and are uniquely determined by the numbers, for $f,g\in\rng F^*,$
\begin{equation}\label{F: wave matrix form}
 \scal{\euE_{\lambda}(H_1)f,w_\pm(\lambda;H_1,H_0)\euE_{\lambda}(H_0)g} := \lim_{y\to 0^+}\frac{y}{\pi}\scal{R_{\lambda\pm iy}(H_1)f,R_{\lambda\pm iy}(H_0)g}.
\end{equation}
Further, the wave matrices \eqref{F: wave matrix} satisfy analogous equalities to \eqref{F: a(z)}, \eqref{F: Ew}, and \eqref{F: wE}, in other words these equalities hold in the case that $y=0.$
\end{prop}
\begin{proof}
The first part of this proof follows a standard method, see e.g. \cite[\S 5.2]{Ya}. 
The formula \eqref{F: wave matrix form} defines a sesquilinear form on $\rng\euE_{\lambda}(H_1)\times \rng\euE_{\lambda}(H_0),$ which we will show is bounded, with bound $\leq 1,$ and hence defines a bounded operator \eqref{F: wave matrix} from the closure of $\rng\euE_{\lambda}(H_0)$ to the closure of $\rng\euE_{\lambda}(H_1).$
For $y>0$ and any $f=F^*\varphi,g=F^*\psi\in\rng F^*,$ 
\begin{align*}
 \frac{y}{\pi}|\scal{R_{\lambda\pm iy}(H_1)f,R_{\lambda\pm iy}(H_0)g}| &\leq \frac{y}{\pi}\|R_{\lambda\pm iy}(H_1)F^*\varphi\|\|R_{\lambda\pm iy}(H_0)F^*\psi\|
\\&= \frac{1}{\pi}\scal{\varphi,\Im T_{\lambda+iy}(H_1)\varphi}^{1/2}\scal{\psi,\Im T_{\lambda+iy}(H_0)\psi}^{1/2}
\\&=\|\euE_{\lambda+iy}(H_1)f\|\|\euE_{\lambda+iy}(H_0)g\|
\end{align*}
Then from \eqref{F: wave matrix form} and since $\lambda\in\Lambda(H_0,F)\cap\Lambda(H_1,F),$ by taking the limit $y\to 0^+$ we obtain
\[
 |\scal{\euE_{\lambda}(H_1)f,w_\pm(\lambda;H_1,H_0)\euE_{\lambda}(H_0)g}| \leq \|\euE_{\lambda}(H_1)f\|\|\euE_{\lambda}(H_0)g\|.
\]
Now we show that the equalities \eqref{F: a(z)}, \eqref{F: Ew}, and \eqref{F: wE} imply analogous equalities for the wave matrices. 
Indeed, for any $\varphi,\psi\in\clK,$ from the definition \eqref{F: wave matrix form} we obtain
\begin{align*}
&\scal{\sqrt{\Im T_{\lambda+i0}(H_1)}\varphi,w_\pm(\lambda;H_1,H_0)\sqrt{\Im T_{\lambda+i0}(H_0)}\psi}
\\&\qquad\qquad= \lim_{y\to 0^+} \scal{\sqrt{\Im T_{\lambda + iy}(H_1)}\varphi,w(\lambda\pm iy;H_1,H_0)\sqrt{\Im T_{\lambda\pm iy}(H_0)}\psi}.
\end{align*}
Then \eqref{F: a(z)}, \eqref{F: Ew}, and \eqref{F: wE}, imply that the left hand side of the above equality is
\begin{align}
 LHS &= \scal{\varphi,\lim_{y\to 0^+} yFR_{\lambda\mp iy}(H_1)R_{\lambda\pm iy}(H_0)F^*\psi} \label{F: a(z) for y=0} 
\\&= \scal{\varphi,(1-T_{\lambda-i0}(H_1)J)\sqrt{\Im T_{\lambda+i0}(H_0)}\sqrt{\Im T_{\lambda+i0}(H_0)}\psi} \label{F: Ew for y=0}
\\&= \scal{\sqrt{\Im T_{\lambda+i0}(H_1)}\varphi,\sqrt{\Im T_{\lambda+i0}(H_1)}(1+JT_{\lambda+i0}(H_0))\psi}. \label{F: wE for y=0}
\end{align}
The equality \eqref{F: a(z) for y=0} implies \eqref{F: a(z)} in the case that $y=0,$ considered as an equality of bounded operators on $\clK.$
Similarly, by the density of the range of $\sqrt{\Im T_{\lambda+i0}(H)}$ in $\hlambda(H),$ the equality \eqref{F: Ew for y=0} implies \eqref{F: Ew} in the case that $y=0,$ considered as an equality of bounded operators from $\hlambda(H_0)$ to $\clK,$ 
and \eqref{F: wE for y=0} implies \eqref{F: wE} in the case that $y=0,$ considered as an equality of bounded operators from $\clK$ to $\hlambda(H_1).$ 
\end{proof}

\begin{thm}\label{T: wave matrix}
Let $F\colon\clH\to\clK$ be a rigging operator and let $H_0,H_1$ and $H_2$ be self-adjoint operators on $\clH$ whose pairwise differences belong to $F^*\clB_{sa}(\clK)F.$ 
The wave matrices satisfy the multiplicative property
\begin{equation}\label{F: mult prop}
 w_\pm(\lambda;H_2,H_0) = w_\pm(\lambda;H_2,H_1)w_\pm(\lambda;H_1,H_0),
\end{equation}
for any $\lambda\in\Lambda(H_0,F)\cap\Lambda(H_1,F)\cap\Lambda(H_2,F).$
In addition, the wave matrices \eqref{F: wave matrix} are unitary and satisfy the equalities $w(\lambda;H_0,H_0) = 1$ and $w^*_\pm(\lambda;H_1,H_0) = w_\pm(\lambda;H_0,H_1).$
\end{thm}
\begin{proof}
To prove \eqref{F: mult prop} it is enough to show that for any $f,g\in\rng F^*$
\begin{multline*}
 \scal{\euE_{\lambda}(H_2)f,w_\pm(\lambda;H_2,H_0)\euE_{\lambda}(H_0)g} \\= \scal{\euE_\lambda(H_2)f,w_\pm(\lambda;H_2,H_1)w_\pm(\lambda;H_1,H_0)\euE_\lambda(H_0)g}.
\end{multline*}
This can be inferred from the following equality of bounded operators on the auxiliary Hilbert space $\clK.$
\begin{align*}
&\sqrt{\Im T_{\lambda+i0}(H_2)}w_\pm(\lambda;H_2,H_0)\sqrt{\Im T_{\lambda+i0}(H_0)} 
\\&\qquad = \lim_{y\to 0^+}\sqrt{\Im T_{\lambda+iy}(H_2)}w(\lambda\pm iy;H_2,H_0)\sqrt{\Im T_{\lambda+iy}(H_0)}
\\&\qquad = \lim_{y\to 0^+}\sqrt{\Im T_{\lambda+iy}(H_2)}w(\lambda\pm iy;H_2,H_1)w(\lambda\pm iy;H_1,H_0)\sqrt{\Im T_{\lambda+iy}(H_0)}
\\&\qquad = \sqrt{\Im T_{\lambda+i0}(H_2)}w_\pm(\lambda;H_2,H_1)w_\pm(\lambda;H_1,H_0)\sqrt{\Im T_{\lambda+i0}(H_0)}.
\end{align*}
Here, the first equality follows from \eqref{F: a(z)} and its analogue in the case $y=0$ (Proposition \ref{P: wave matrix}),
the second equality uses the multiplicative property of the off-axis wave matrix, 
while the final equality follows from \eqref{F: Ew} and \eqref{F: wE} and their analogues in the case $y=0.$ 

The multiplicative property can now be used to prove the remaining properties. 
Firstly, it follows easily from the definition \eqref{F: wave matrix form} that $w_\pm(\lambda,H_0,H_0) = 1.$
Combining this with the multiplicative property, we have
\begin{gather*}
w_\pm(\lambda;H_1,H_0)w_\pm(\lambda;H_0,H_1) = w_\pm(\lambda;H_1,H_1) = 1,
\\
w_\pm(\lambda;H_0,H_1)w_\pm(\lambda;H_1,H_0) = w_\pm(\lambda;H_0,H_0) = 1.
\end{gather*}
From these equalities and the fact that $\|w_\pm(\lambda)(H_1,H_0)\|\leq 1,$ it can be inferred that the wave matrices are unitary and $w_\pm^*(\lambda;H_1,H_0) = w_\pm(\lambda;H_0,H_1).$
\end{proof}

\begin{thm}\label{T: stat formula}
Let $F\colon\clH\to\clK$ be a rigging operator and let $H_0$ and $H_1$ be self-adjoint operators such that $H_1 - H_0 = F^*JF$ for some $J\in\clB_{sa}(\clK).$ 
The off-axis scattering matrix and the scattering matrix itself, defined respectively by
\begin{gather*}
 S(z;H_1,H_0) := w^*(z;H_1,H_0)w(\bar{z};H_1,H_0), \quad z=\lambda+iy,\, y>0,
\quad \text{and}
\\ S(\lambda;H_1,H_0) := w^*_+(\lambda;H_1,H_0)w_-(\lambda;H_1,H_0), \quad \lambda\in\Lambda(H_0,F)\cap\Lambda(H_1,F),
\end{gather*}
both satisfy the formula:
\begin{equation}\label{F: stat formula}
 S(z;H_1,H_0) = 1 - 2i \sqrt{\Im T_{z}(H_0)}J(1 - T_z(H_1)J)\sqrt{\Im T_{z}(H_0)},
\end{equation}
where for the scattering matrix, put $z=\lambda$ on the left and $z=\lambda+i0$ on the right.
\end{thm}

As a consequence of \eqref{F: stat formula}, if $\lambda$ belongs to the intersection $\Lambda(H_0,F)\cap\Lambda(H_1,F),$ then the limit $S(\lambda+i0;H_1,H_0)$ of the off-axis scattering matrix is equal to  $S(\lambda;H_1,H_0)\oplus 1$ acting on $\clK = \hlambda(H_0)\oplus\hlambda(H_0)^\perp.$

\begin{proof}
Let $z=\lambda+iy,$ $y\geq 0,$ where we assume $\lambda\in\Lambda(H_0,F)\cap\Lambda(H_1,F)$ if $y=0.$
Using \eqref{F: Ew} and \eqref{F: wE}, or Proposition \ref{P: wave matrix} for their analogues in the case $y=0,$ we obtain
\begin{align*}
 \sqrt{\Im T_{z}(H_0)}S(z;H_1,H_0)\sqrt{\Im T_{z}(H_0)} 
&= (1 + T_{\bar{z}}(H_0)J)\Im T_{z}(H_1)(1 + JT_{\bar{z}}(H_0))
\\&= \Im T_{z}(H_0)(1 - JT_{z}(H_1))(1 + JT_{\bar{z}}(H_0)).
\end{align*}
While it follows from the second resolvent identity that 
\begin{align*}
 (1 - JT_{z}(H_1))(1 + JT_{\bar{z}}&(H_0)) \\&= 1 - JT_z(H_1) + J(1 - T_z(H_1)J)T_{\bar{z}}(H_0)
\\&= 1 - J(1 - T_z(H_1)J)T_z(H_0) + J(1 - T_z(H_1)J)T_{\bar{z}}(H_0)
\\&= 1 -2iJ(1 - T_z(H_1)J)\Im T_z(H_0).
\end{align*}
Therefore, for any $f=F^*\varphi,g=F^*\psi\in\rng F^*$ we get
\begin{align*}
\langle\euE_z(H_0)f, S(z;H_1,H_0)&\euE_{z}(H_0)g\rangle 
\\& = \scal{\varphi,\Im T_{z}(H_0)\left[1 -2iJ(1 - T_z(H_1)J)\Im T_z(H_0)\right]\psi}
\\& = \scal{\euE_z(H_0)f, \left(\ldots\right)\euE_{z}(H_0)g},
\end{align*}
where $(\ldots)$ stands for the right hand side of \eqref{F: stat formula}. 
This implies \eqref{F: stat formula} by the density of the range of $\euE_{z}(H_0),$ which is dense in $\clK$ in the case that $y>0$ and dense in $\hlambda(H_0)$ in the case that $y=0.$
\end{proof}

\smallskip
The wave operators and the scattering operator can be built up from the wave matrices and scattering matrix.
For example, with $\Lambda:=\Lambda(H_0,F)\cap\Lambda(H_1,F),$ the (partial) wave operators can be defined by
\begin{equation}\label{F: WO defn}
 W_\pm(H_1,H_0;\Lambda) = \int^\oplus_\Lambda w_\pm(\lambda;H_1,H_0) \, d\lambda.
\end{equation}
Since the direct integral $\euH(H)$ is isomorphic to the subspace $E(\Lambda)\clH\subset\clH^{(a)}(H)$ by Theorem \ref{T: spectral thm}, it follows that the wave operators \eqref{F: WO defn} can be considered as operators 
\begin{equation}\label{F: WO on H}
W_\pm(H_1,H_0;\Lambda)\colon E^{H_0}(\Lambda)\clH \to E^{H_1}(\Lambda)\clH.
\end{equation}
As a result of Theorem \ref{T: wave matrix}, the wave operators \eqref{F: WO on H} are unitary, satisfy the equalities $W_\pm(H_0,H_0;\Lambda) = 1$ and $W_\pm^*(H_1,H_0;\Lambda) = W_\pm(H_0,H_1;\Lambda),$ and admit the multiplicative property $W_\pm(H_2,H_0;\Lambda) = W_\pm(H_2,H_1;\Lambda)W_\pm(H_1,H_0;\Lambda)$ if $\Lambda = \Lambda(H_0,F)\cap\Lambda(H_1,F)\cap\Lambda(H_2,F).$
If $\Lambda$ is a full set, then these properties imply:
\begin{thm}\label{T: Kato-Rosenblum}
Let $F\colon\clH\to\clK$ be a rigging operator and let $H_0$ and $H_1$ be self-adjoint operators such that $H_1-H_0\in F^*\clB_{sa}(\clK)F.$ 
Suppose the Limiting Absorption Principle holds for $H_0$ and $H_1.$ 
Then their absolutely continuous parts are unitarily equivalent. 
\end{thm}

\begin{thm}
Under the premise of Theorem \ref{T: Kato-Rosenblum}, the wave operators \eqref{F: WO on H} coincide with their classical time-dependent definition. 
That is, the strong operator limits
\[
 s\mbox{-}\hspace{-0.6em}\lim_{t \to \pm \infty} e^{itH_1}e^{-itH_0}P^{(a)}(H_0),
\]
where $P^{(a)}(H_0)$ is the projection onto the absolutely continuous subspace $\clH^{(a)}(H_0),$ exist and coincide with the wave operators defined by \eqref{F: WO on H} and extended to $\clH$ as zero on $E^{H_0}(\mbR\setminus\Lambda).$
\end{thm}
\noindent The proof of this theorem uses a standard method and will not be used further, so we only briefly sketch a proof.
The Limiting Absorption Principle implies the existence of the weak wave operators (see e.g. \cite[Theorem 5.3.2]{Ya}). 
It can be checked that the weak wave operators coincide with the wave operators defined by formula \eqref{F: WO defn}.
Combined with the multiplicative property of the wave operators, this implies the existence of the strong wave operators and that they coincide with \eqref{F: WO on H} (see e.g. \cite[Theorem 2.2.1]{Ya}).

\subsection{SSF and the scattering matrix}
In this subsection we reinstate the resolvent comparable assumption \eqref{F: resolvent comparable};
throughout, $\clA(F)$ will be as defined in Section \ref{S: sSSF}. 
A piecewise analytic path $H_r$ in $\clA(F),$ will be written as $H_r = H_0 + F^*J_rF,$ where $J_r$ is a piecewise analytic path in $\clB_{sa}(\clK)$ such that $J_0 = 0.$
By $\clU_1(\clH) := \{U \in 1 + \clL_1(\clK) : U$ is unitary$\},$ we denote the group of unitary operators differing from 1 by a trace class operator, with the complete metric $(U,V) \mapsto \|U - V\|_1.$

Let $z=\lambda+iy,$ $y\geq 0,$ where if $y=0$ we assume that $\lambda$ belongs to the set $\Lambda(H_0,F;\clL_1).$ 
Consider the stationary formula for $S(z;H_r,H_0)$:
\begin{equation}\label{F: Hr stat formula}
S(z;H_r,H_0) = 1 - 2i\sqrt{\Im T_{z}(H_0)}J_r(1 + T_{z}(H_0)J_r)^{-1}\sqrt{\Im T_{z}(H_0)}, 
\end{equation}
which follows from Theorem \ref{T: stat formula} and the second resolvent identity.
The operator $\Im T_z(H_0)$ is trace class and hence $S(z;H_r,H_0)$ belongs to $\clU_1(\clK),$ provided that \eqref{F: Hr stat formula} holds.
If $y>0,$ then it holds for all $r.$
On the other hand if $y=0,$ then it holds as long as $r$ is not a resonance point of the path $H_r.$

The stationary formula \eqref{F: Hr stat formula} allows us to consider the function
\begin{equation}\label{F: S(r)}
r \mapsto S(z;H_r,H_0) \in \clU_1(\clK).
\end{equation} 
Supposing that $H_r$ is analytic, the analytic Fredholm alternative implies that the factor $(1 + T_{z}(H_0)J_r)^{-1}$ is meromorphic. 
Hence in a neighbourhood of the real axis \eqref{F: S(r)} is a meromorphic function. 
If $y>0,$ then \eqref{F: S(r)} is in fact holomorphic in a neighbourhood of $\mbR.$
If $y=0,$ then although the factor $(1 + T_{\lambda+i0}(H_0)J_r)^{-1}$ has poles at real resonance points $r\in R(\lambda;\{H_r\}),$ the scattering matrix $S(\lambda; H_r, H_0)$ is unitary for all non-resonant real $r.$ 
Therefore \eqref{F: S(r)} must be bounded on the real axis and so admits analytic continuation there.

\begin{thm}\label{T: S'(r)}
Let $H_r$ be an analytic path in $\clA(F)$ and let $z=\lambda+iy,$ $y\geq 0,$ where in the case that $y=0$ it is assumed that $\lambda\in\Lambda(H_0,F;\clL_1).$ 
Then at any non-resonant real $r,$ the derivative of $S(z;H_r,H_0)$ with respect to $r$ is given by \eqref{F: S ODE}, where the derivative is taken in the norm of $\clL_1(\clK)$
(while the formula \eqref{F: S ODE} is given for the scattering matrix, an analogous formula holds in the case that $y>0).$
\end{thm}
\begin{proof}
We consider the case when $y = 0;$ 
in case $y>0,$ the formula \eqref{F: S ODE} holds for any real $r$ and the calculation is identical.
The derivative of the meromorphic function $r\mapsto J_r(1+T_z(H_0)J_r)^{-1}$ appearing in the stationary formula \eqref{F: Hr stat formula} can be calculated for any non-resonant $r,$ yet since $J_0 = 0$ its derivative at 0 is simply $\dot{J}_0.$ 
Hence it follows from the stationary formula that
\begin{equation}\label{F: ISM}
 \frac d{dr} \Big|_{r=0} S(\lambda; H_{r}, H_0)  = - 2 i \sqrt{\Im T_{\lambda+i0}(H_0)}\dot{J}_0\sqrt{\Im T_{\lambda+i0}(H_0)}
\end{equation}
Following from the properties of the wave matrix, the scattering matrix satisfies 
\[
 S(\lambda;H_{r+h},H_0) = w_+(\lambda;H_0,H_r)S(\lambda;H_{r+h},H_r)w_+(\lambda;H_r,H_0)S(\lambda;H_r,H_0)
\]
for any non-resonant $r$ and $r+h.$	
If $r$ is non-resonant, then since the resonance set $R(\lambda;\{H_r\})$ is discrete, so is $r+h$ for small $h.$
Thus
\begin{multline*}
 \frac d{dr} S(\lambda; H_{r}, H_0) 
\\= w_+(\lambda; H_0,H_r)\frac d{dh} \Big|_{h=0}S(\lambda; H_{r+h}, H_r)w_+(\lambda; H_r,H_0)S(\lambda; H_r, H_0)
\end{multline*}
and the proof is completed by substituting \eqref{F: ISM}.
\end{proof} 

It follows from Theorem \ref{T: S'(r)} and the unitarity of the scattering matrix that the trace class valued function
\begin{equation*}
r\mapsto w_+(\lambda; H_0,H_r)\sqrt{\Im T_{\lambda+i0}(H_r)}\dot{J}_r\sqrt{\Im T_{\lambda+i0}(H_r)}w_+(\lambda; H_r,H_0),
\end{equation*}
although only defined for non-resonant values of $r,$ admits analytic continuation to the real axis.
Taking its trace shows that the function 
$
 r\mapsto \Tr(\dot{J}_r\Im T_{\lambda+i0}(H_r))
$
also admits analytic continuation to the real axis, as was required in the proof of Theorem \ref{T: sSSF = res ind}.

Theorem \ref{T: S'(r)} implies the ordered exponential representation \eqref{F: S = Texp}, where $H_r$ is a piecewise analytic path in $\clA(F)$ whose endpoints are not resonant at $\lambda.$
There is an analogous representation for the off-axis scattering matrix.
Because the ordinary differential equation \eqref{F: S ODE} can be considered in the trace class norm, we also obtain the following theorem from properties of the ordered exponential (see \cite[Appendix]{Az3v6}).

\begin{thm}\label{T: B-K}
Let $H_r$ be a piecewise analytic path in $\clA(F).$ 
For $z=\lambda+iy,$ $y\geq 0,$ where it is assumed that $\lambda$ belongs to the intersection $\Lambda(H_0,F;\clL_1)\cap\Lambda(H_1,F;\clL_1)$ if $y=0,$ there is the formula
\begin{equation}\label{F: B-K}
 \det S(z; H_1, H_0) = \exp\left(-2i \int_0^1 \Tr(\dot{J}_r\Im T_{z}(H_r))\,dr\right).
\end{equation}
\end{thm}

This theorem can be interpreted as a variant of the Birman-Krein formula. 
The Birman-Krein formula \eqref{F: det S = exp(-2pi i xi)} can be recovered as follows. 
For $y>0,$ the formula \eqref{F: B-K} can be rewritten as 
\begin{equation}\label{F: B-K y>0}
\det S(z; H_1, H_0) = e^{- 2\pi i \,\xi(z;H_1,H_0)},
\end{equation}
where $\xi(z;H_1,H_0)$ is the smoothed SSF \eqref{F: smoothed SSF}, whose limit as $y\to 0^+$ is a.e. equal to the SSF $\xi(\lambda;H_1,H_0).$
On the other hand, it follows from the stationary formula that for any $\lambda$ from the set $\Lambda(H_0,F; \clL_1) \cap \Lambda(H_1,F; \clL_1),$ the off-axis scattering matrix $S(\lambda+iy;H_1,H_0)$ converges in $\clU_1(\clK)$ to $S(\lambda;H_1,H_0)\oplus 1$ as $y\to 0^+.$
Therefore, taking the limit of \eqref{F: B-K y>0} as $y\to 0^+$ results in the equality \eqref{F: det S = exp(-2pi i xi)}, which holds for a.e. $\lambda\in\Lambda(H_0,F; \clL_1) \cap \Lambda(H_1,F; \clL_1).$

In addition to \eqref{F: det S = exp(-2pi i xi)}, Theorem \ref{T: B-K} explicitly gives the formula
\[
\det S(\lambda; H_1, H_0) = \exp\left(-2i \int_0^1 \Tr(\dot{J}_r\Im T_{\lambda+i0}(H_r))\,dr\right).
\]
for any $\lambda\in\Lambda(H_0,F;\clL_1)\cap\Lambda(H_1,F; \clL_1).$ 
This may be rewritten as \eqref{F: det S = exp(-2pi i xia)} in view of Theorem \ref{T: a.c.SSF}.
Combining \eqref{F: det S = exp(-2pi i xi)} and \eqref{F: det S = exp(-2pi i xia)} gives the equality 
\[
e^{-2\pi i \,\xi^{(s)}(\lambda;\{H_r\})} = 1,
\] 
which holds for a.e. $\lambda\in\Lambda(H_1,F;\clL_1)\cap\Lambda(H_0,F;\clL_1).$ 	
Thus we have again proved the integer-valuedness of the singular SSF, 
this time along any piecewise analytic path $H_r.$
To repeat and summarise:
\begin{cor}
Let $H_r$ be a piecewise $C^1$ path in $\clA(F).$ Then for a.e. $\lambda\in\mbR,$ 
$e^{-2\pi i \xi(\lambda;H_1,H_0)} = \det S(\lambda;H_1,H_0) = e^{-2\pi i \xi^{(a)}(\lambda;\{H_r\})}$ and $\xi^{(s)}(\lambda;\{H_r\})\in\mbZ.$
\end{cor}

\subsection{Singular SSF and singular \texorpdfstring{$\mu$}{mu}-invariant}
Theorem \ref{T: B-K} allows the equality of the singular SSF and the so called singular $\mu$-invariant, Theorem \ref{T: sSSF and mu inv} below, to be proved following (a simplified version of) the argument appearing in \cite{Az3v6}. 
We conclude by sketching the proof.

Let $U = U(t)$ be a continuous path of operators in $\clU_1(\clK).$ 
The eigenvalues of $U$ can be continuously enumerated, in the sense that there exists a sequence of continuous functions $\lambda_j$ such that for all $t$ the multiset $\{\lambda_1(t),\lambda_2(t),\ldots\}^*$ coincides with the spectrum of $U(t),$ counting multiplicities of all points except 1.
The functions $\lambda_j$ can be used to define the spectral flow of the path $U.$
Here we only need to consider paths which begin (or end) at 1, so suppose $U(0)=1.$
Then, lifting the functions $\lambda_j$ on the circle to the functions $\theta_j$ on $\mbR$ with $\theta_j(0)=0,$ we can count the number of times the eigenvalues of $U(t)$ cross a point $e^{i\theta},$ $\theta\in(0,2\pi),$ using the formula
\begin{equation}\label{F: mu defn}
\mu(\theta;U) := \sum_{j=1}^\infty \left\lfloor \frac{\theta_j(1) - \theta}{2\pi} \right\rfloor,
\end{equation}
where $\lfloor x\rfloor$ is the integer part of $x.$
If $U(0)=U(1)=1,$ then $\mu(\theta;U)$ does not depend on $\theta.$
The spectral flow $\mu$ is a homotopy invariant: if two paths $U_1,U_2\in\clU_1(\clK)$ are homotopic relative to their endpoints, then $\mu(\theta;U_1)=\mu(\theta;U_2).$ 
It is also path additive: if $U_1\sqcup U_2$ is the concatenation of $U_1$ and $U_2,$ then $\mu(\theta;U_1\sqcup U_2) = \mu(\theta;U_1) + \mu(\theta;U_2).$

For a path $U$ in $\clU_1(\clK)$ beginning at 1, it can be shown using the formula \eqref{F: mu defn} that
\[
\xi(U) := -\frac{1}{2\pi}\int_0^{2\pi} \mu(\theta;U) \,d\theta = -\frac{1}{2\pi}\sum_{j=1}^{\infty}\theta_j(1).
\]
Further, letting $U_t$ denote the restriction of the path $U$ to the interval $[0,t],$ it can be shown that $t\mapsto \xi(U_t)$ is continuous.
On the other hand, for a unitary operator $U(t)$ from $\clU_1(\clK)$ with eigenvalues $e^{i\theta_j(t)},$ $j=1,2,\ldots,$ we have
\[
\det U(t) = \prod_{j=1}^{\infty} e^{i\theta_j(t)} = \exp\left(i\sum_{j=1}^\infty\theta_j(t)\right).
\]
Therefore, for any path $U$ in $\clU_1(\clK)$ which begins at 1, there is the formula
\begin{equation}\label{F: log det}
\xi(U_t) = -\frac{1}{2\pi i} \log\det U(t),
\end{equation}
where the branch of the logarithm is chosen so that the right hand side is continuous.

Let $H_0,H_1$ be two self-adjoint operators from $\clA(F).$
For any $\lambda$ from the set $\Lambda(H_0,F;\clL_1)\cap\Lambda(H_1,F;\clL_1),$ the scattering matrix $S(\lambda;H_1,H_0)$ can be naturally connected with the identity in two different ways. 
One way is to send the imaginary part of $\lambda + i0$ from 0 to $+\infty.$
Let $U_1$ be the path
\[
U_1\colon [-\infty,0]\ni y \mapsto S(\lambda-iy;H_1,H_0) \in \clU_1(\clK).
\]
Another way is to send $H_1$ to $H_0$ along a piecewise analytic path $H_r.$ 
Let $U_2$ be the path
\[
U_2\colon [0,1] \ni r \mapsto S(\lambda;H_r,H_0) \in \clU_1(\clK).
\]
Pushnitski's $\mu$-invariant \cite{Pus} is given by 
$
\mu(\theta,\lambda;H_1,H_0) := \mu(\theta; U_1),
$
whereas the absolutely continuous $\mu$-invariant \cite{Az3v6}, is given by
$
\mu^{(a)}(\theta,\lambda;\{H_r\}) := \mu(\theta; U_2).
$
The difference $\mu^{(s)}(\lambda;\{H_r\}) := \mu(\theta;U_1) - \mu(\theta;U_2),$ which does not depend on $\theta,$ is the {\em singular $\mu$-invariant.}

Applying formula \eqref{F: log det} to the path $U_1$ and using Theorem \ref{T: B-K}, we obtain a formula for the smoothed SSF:
\[
\xi(\lambda+iy;H_1,H_0) = -\frac{1}{2\pi}\int_0^{2\pi} \mu(\theta,\lambda+iy;H_1,H_0) \,d\theta,
\]
where $\mu(\theta,\lambda+iy;H_1,H_0)$ denotes the spectral flow of the path $[-\infty,-y]\ni t \mapsto S(\lambda-it;H_1,H_0).$
Since the smoothed SSF converges a.e. to the SSF $\xi(\lambda;H_1,H_0)$ by \eqref{F: smoothed SSF converges to SSF}, it follows that 
\begin{equation}\label{F: SSF as mu-inv}
\xi(\lambda;H_1,H_0) = -\frac{1}{2\pi}\int_0^{2\pi} \mu(\theta,\lambda;H_1,H_0) \,d\theta,
\end{equation}
for a.e. $\lambda\in\Lambda(H_0,F;\clL_1)\cap\Lambda(H_1,F;\clL_1).$

Applying formula \eqref{F: log det} to the path $U_2$ and using Theorem \ref{T: B-K}, we find a formula for the absolutely continuous SSF:
\begin{equation}\label{F: a.c. SSF as a.c. mu-inv}
\xi^{(a)}(\lambda;\{H_r\}) = -\frac{1}{2\pi}\int_0^{2\pi} \mu^{(a)}(\theta,\lambda;\{H_r\}) \,d\theta.
\end{equation}
Combining the equalities \eqref{F: SSF as mu-inv} and \eqref{F: a.c. SSF as a.c. mu-inv}, we find that for a.e. $\lambda\in\Lambda(H_0,F;\clL_1)\cap\Lambda(H_1,F;\clL_1),$
\begin{align*}
\xi^{(s)}(\lambda;\{H_r\}) &= \xi(\lambda;H_1,H_0) - \xi^{(a)}(\lambda;\{H_r\})
\\ &= -\frac{1}{2\pi}\int_0^{2\pi} \left( \mu(\theta,\lambda;H_1,H_0) - \mu^{(a)}(\theta,\lambda;\{H_r\})\right) \,d\theta
\\ &= -\frac{1}{2\pi}\int_0^{2\pi} \mu^{(s)}(\lambda;\{H_r\}) \,d\theta
\\ &= -\mu^{(s)}(\lambda;\{H_r\}).
\end{align*}

\begin{thm}\label{T: sSSF and mu inv}
Let $H_r$ be a piecewise analytic path of self-adjoint operators from the affine space $\clA(F).$
Then for a.e. $\lambda$ from the full set $\Lambda(H_0,F;\clL_1)\cap\Lambda(H_1,F;\clL_1)$ there is the equality
$
 \xi^{(s)}(\lambda;\{H_r\}) = - \mu^{(s)}(\lambda;\{H_r\}).
$
\end{thm}

\end{document}